\documentclass{amsart}
\usepackage{amsfonts,amscd, amssymb}

\newtheorem{theorem}{Theorem}[section]
\newtheorem{lemma}[theorem]{Lemma}
\newtheorem{corollary}[theorem]{Corollary}

\theoremstyle{remark}

\theoremstyle{definition}

\numberwithin{equation}{section}
\makeatother

\DeclareMathOperator{\Cdb}{{\mathbb C}}

\begin{document}

\title{A noncommutative  Bishop peak interpolation-set theorem}
\author{David P. Blecher}
\address{Department of Mathematics, University of Houston, Houston, TX
77204-3008, USA}
%ADD
\email[David P. Blecher]{dpbleche@central.uh.edu}

%\date{9/8/2022}  
\thanks{DB  is supported by a Simons Foundation Collaboration Grant.  }

\begin{abstract}  We prove a noncommutative version of  Bishop's peak interpolation-set theorem.  \end{abstract}

\maketitle

\section{Introduction}

For us, an operator algebra is a norm closed algebra of operators on a Hilbert space, or equivalently a closed subalgebra $A$ of a $C^*$-algebra $B$.
In this paper we shall assume for simplicity that $A$ and $B$ share a common identity element. 
In `noncommutative peak interpolation' as surveyed briefly in \cite{Bnpi} say, one generalizes the classical peak interpolation theory to the setting of
operator algebras, using Akemann's noncommutative topology (see \cite{Ake2} and e.g.\ other references  in \cite{Bnpi}).
In classical peak 
interpolation\footnote{When we use the term `peak interpolation' we mean in this sense.  Others sometimes use this term to refer to what we call peak interpolation-sets below.}
 the setting is a subalgebra $A$ of $B = C(K)$, the continuous scalar functions on a compact Hausdorff space $K$, and one tries to build
functions in $A$ which have prescribed values or behaviour on a fixed closed subset $E \subset K$ (or on 
several disjoint subsets).   The sets that  `work' for this are the $p$-sets, namely the  intersections of {\em peak sets}, where the latter term means a 
set of form $f^{-1}(\{ 1 \})$ for $f \in A, \Vert f \Vert = 1$ (in the separable case, they are just the $p$-sets).   
A typical `peak interpolation result',
says that if $f
\in C(K)$ is strictly positive, and $E$ is a $p$-set, then the continuous functions  $g$ on $E$  which are restrictions of 
functions in $A$, and which are dominated in modulus 
by the `control function' $f$ on $E$, have extensions $h$ in $A$ satisfying $|h| \leq f$ on 
all of $K$ (see 
 e.g.\ II.12.5 in \cite{Gam}; there are nice pictures illustrating this result in \cite{Bnpi}).    We shall call this the `Gamelin-Bishop theorem' below.   
 The special case of this where $g = 1$ in fact characterizes $p$-sets among the closed subsets of 
$K$ (e.g.\ see \cite{Gamr}).    Combining this with a result of Glicksberg \cite[Theorem II.12.7]{Gam}), one obtains that the $p$-sets are the closed sets $E$ in  
$K$ with $\mu_E \in A^\perp$ for all measures $\mu \in A^\perp$.   Equivalently, if and only if the characteristic function $\chi_E$ is in $A^{\perp \perp}$.

An {\em interpolation-set} for $A$ is a closed set $E \subset K$ such that $A_{|E} = C(E)$.    A  {\em  peak interpolation-set} (resp.\ $p$-{\em interpolation set})
is a peak (resp.\ $p$-) set which is also an  interpolation-set.    In the light of the above it is clear and obvious that the $p$-interpolation sets may be characterized
by the appropriate variant of the Gamelin-Bishop theorem above: if  $f
\in C(K)$ is strictly positive, then the continuous functions  $g$ on $E$ which are dominated in modulus
by the `control function' $f$ on $E$, have extensions $h$ in $A$ satisfying $|h| \leq f$ on 
all of $K$.   There are other characterizations of $p$-interpolation sets, e.g.\ as the closed sets  in  
$K$ with $\mu_E =0$ for all measures $\mu \in A^\perp$.   By basic measure theory, 
the latter is equivalent to: $|\mu|(E) = 0$ for all such $\mu$.   
By the {\em Bishop peak interpolation-set theorem} we shall mean the result that Bishop proved in 
\cite{Bish}: namely that  if $\mu_E =0$ for all measures $\mu \in A^\perp$ then the extension theorem stated a few lines back involving $f, g,h,$ holds.

A special case of interest  is 
when $f$ above is $1$, and when this case  is applied to the 
disk algebra, together with the F $\&$ M Riesz theorem,  one obtains the famous Rudin-Carleson theorem
(see e.g.\ II.12.6 in \cite{Gam}).   

The noncommutative analogue of $p$-sets are the $p$-projections.   Analogously to the classical case above they have been characterized as 
the infima (`meet') of peak projections, or as the closed projections in $B^{**}$ that lies in $A^{\perp \perp}$, if $A$ is a unital subalgebra of a $C^*$-algebra $B$ (see the start of Section 2
for references).
See \cite{Bnpi} for a survey of part of the `noncommutative peak set'  theory, and see also references therein for dozens of our other results.  
More recently Davidson and Clou\^atre and Hartz and others, have studied forms of 
 noncommutative peak interpolation-sets in specific classes of operator algebras (see e.g.\ \cite{DH,CD,CD2,CT,CTh} and references therein).    Their aim is often to generalize to such classes the classical theory of interpolation sets 
 for the ball algebra (see \cite[Chapter 10]{Rudin}), the Rudin-Carleson theorem, etc.   This work has strong connections to our general `noncommutative peak interpolation' theory (from \cite{Bnpi} and references therein, or in the present paper), indeed one of the main theorems in \cite{DH} follows quickly from our more general theory as we mention at the end of this Introduction.
 Exploring these connections raises many interesting questions and should lead to further important progress.
 
The following is the unital case of a very general noncommutative variant of the Gamelin-Bishop theorem \cite[Theorem 3.4]{Bnpi}:  

\begin{theorem} \label{peakthang22}    \cite{Bnpi}  Suppose that $A$ is a subalgebra of a unital $C^*$-algebra $B$, with $1_B \in A$.  Suppose
 that  $q$ is a closed projection in $B^{**}$ that lies in $A^{\perp \perp}$.
   If $b \in A$ with $b q= qb$,
 and $q b^* b q  \leq q d$ for an invertible positive $d \in B$ which commutes with $q$,
then  there exists an element $a \in A$ with $a q =  q a = b q$, and $a^* a \leq d$.
\end{theorem}

This result may fail without the `commuting hypothesis'  $b q= qb$, and even in the case $d = 1$.   See \cite[Corollary 2.4]{BRII} for an example of a peak projection $q$
and $b \in A$ with $\| b q \| \leq 1$, but there is no $a \in {\rm Ball}(A)$ with $aq = bq$.   Thus as stated after that result,  ``Clearly the way to proceed from this point, in noncommutative peak interpolation, is to insist on a commutativity assumption of [this] type".    Let us  apply this principle in an attempt to find a noncommutative version of the 
classic Bishop peak interpolation-set theorem above.   We first note that \cite[Propositions 3.2 and 3.4]{Hay} may be regarded as such, however there is an $\epsilon$ parameter there that one does not see in Bishop's result.
Our next observation is that simple noncommutative versions of the Bishop peak interpolation-set theorem  follow immediately from Theorem \ref{peakthang22}  and  \cite[Proposition 3.4]{Hay}.
Indeed right at the end of \cite{Bnpi} we alluded to this: ``Finally, we remark that simple Tietze theorems of the flavour of the Rudin-Carleson theorem mentioned on the first page, follow from our interpolation theorems by adding a hypothesis of the kind in Proposition 3.4 of'' \cite{Hay}. 
For example suppose that $q$ is in the center of $B^{**}$. 
 If $A^\perp \subset (q B)^\perp$, and $b \in B$  then  \cite[Proposition 3.4]{Hay} provides $a \in A$ with 
$aq = bq = q a$ and $q b^* b q = q a^* a q$.  Thus we may apply Theorem \ref{peakthang22}  to obtain  $g \in A$ with $g q =  q g = b q = aq$, and $g^* g \leq d$ as desired.  
However the assumption that $q$ is central in $B^{**}$ is very strong.   If $q$ is not central but $bq = qb$ then one may modify the argument above, but we were unable to see at the time of \cite{Bnpi} how to get the desired conclusion without adding a strong or  unappealing hypothesis.  For example if $A$ is commutative (which is the case in many results discussed in e.g.\ \cite{CD,CD2,DH}, and of course it does not imply that $B$ is also commutative) then this modified argument works (see the proof of Corollary \ref{peakthstr} (5) below). 

In any case,  the desired noncommutative Bishop peak interpolation-set theorem  suggested by the above would be      a characterization of the 
closed projections $q$  in $A^{\perp \perp}$ with the following property:
If $b \in B$ with $b q= qb$,  and $q b^* b q  \leq q d$ for an invertible positive $d \in B$ which commutes with $q$,
then  there exists an element $a \in A$ with $a q =  q a = b q$, and $a^* a \leq d$.   
In the present paper we supply such a theorem.

 Turning to notation, the reader is referred for example to \cite{BLM,BHN,BRI,Bnpi} for more details on
some of the topics below if needed.  We will use silently the fact from basic analysis that
$X^{\perp \perp}$ is the weak* closure in $Y^{**}$ of
a subspace  $X \subset
Y$, and is isometric to $X^{**}$.   For us a {\em projection}
is always an orthogonal projection.   If $A$ is a unital 
operator algebra then its 
second dual $A^{**}$ is  a unital operator algebra with its (unique)
Arens product, this is also the product inherited from the von Neumann
algebra $B^{**}$ if $A$ is a subalgebra of a $C^*$-algebra $B$. Via semicontinuity, it is natural 
to declare a projection  $q \in B^{**}$ to be {\em open} if it is a increasing (weak*) limit of 
positive elements in $B$, and closed if its `perp'  $1-q$ is {\em open}.   We write $\chi_E$ for
the characteristic function of $E$.   In the 
case that $B = C(K)$, and $E$ is an open or  closed set
 in $K$, the projection $q = \chi_E$  may be viewed as an element of $C(K)^{**}$ in a natural way since 
$C(K)^*$ is a certain space of measures on $K$.     Thus if $B = C(K)$
the open or closed projections are precisely the   characteristic functions of  open or  closed sets.

{\bf Acknowledgements:} We thank Ken Davidson and Michael Hartz for several stimulating conversations, some of which are explicitly referenced 
 in the body of this paper.   This project was initiated during the IWOTA Conference at Lancaster in 2021 when Davidson asked us
 if we had a noncommutative Bishop peak interpolation-set theorem.  
 We replied by mentioning the observation in the second last paragraph and the quote from the end of \cite{Bnpi} stated there, and this led to a correspondence in 2021  involving a 
 variant of the present manuscript.  As we said above, the connections between our `noncommutative peaking theory' and the  work of Clou\^atre, Davidson, Hartz, Timko and others, raise many interesting questions to explore further.  For example using our noncommutative peak theory (and results of Hay, etc) and some additional work
 to prove some operator algebraic peak interpolation results in \cite{DH}.  Indeed in the course of our discussions Michael Hartz developed some ideas for the latter (and Michael and Ken had many other helpful perspectives and ideas).
Shortly thereafter  Raphael Clou\^atre was able to find a very short argument showing how Theorems 1.4 and 8.1 in \cite{DH}
follow from our earlier peak interpolation theorem (surveyed above or  in \cite{Bnpi}).  We thank him for his careful 
 explanation.  

\section{On the  `Bishop peak interpolation-set' result}

In this paper $A$ is a subalgebra of a unital $C^*$-algebra $B$ with $1_B \in A$.   Noncommutative peak sets for $A$ 
were introduced in the thesis of Damon Hay \cite{Hay}.
 A projection $p \in A^{**}$ is called
open with respect to $A$ or $A$-open if   there exists a net $x_t \in A \cap p A^{**} p$ with $x_t \to p$ weak* (see  \cite[Section 2]{BHN}).
It is $A$-closed if $1-p$ is $A$-open.    These coincide with  the projections in $B^{**}$ that are open or closed in the $C^*$-algebra sense,
that also lie in $A^{\perp \perp}$ \cite[Theorem 2.4]{BHN}. They also coincide with the $p$-projections for $A$   (see e.g.\ \cite[Theorem 1.2]{BRI} and \cite[Theorem 3.4]{BNcp}; these results have also been generalized to Jordan operator algebras 
  by the author and Neal).

\begin{lemma} \label{qinz} Suppose that $A$ is a unital operator algebra, 
a subalgebra of a unital $C^*$-algebra $B$.
 Let $q$ be a closed projection in $A^{**}$.  If $A_0$ is the subalgebra
 $\{ a \in A : aq = qa \}$ then $q$ is a closed projection in $A_0^{**}$
 (and in $B_0^{**}$). 
Moreover, $q B_0$ and $q A_0$ are norm closed.
\end{lemma}
\begin{proof}  
Indeed suppose that  $x_t \in A \cap (1-q) A^{**} (1-q)$ with $x_t \to 1-q$ weak*. 
  Then $q x_t = x_t q = 0$ so that $1-x_t \in A_0$ and $1-x_t \to q$ weak*.
So $q$ and $1-q$ are in $A_0^{\perp \perp}$ and $1-q$ is an open projection for $A_0$
by the latter mentioned definition.

It follows from 
\cite[Proposition 3.1]{Hay} with $X = A_0$ that $q A_0$ is closed.   Indeed if $\varphi \in A_0
^{\perp}$ 
then since $qA_0 \subset A_0^{\perp \perp}$ it follows that $\varphi \in (q A_0)_\perp$.   So $q A_0$ is closed by \cite[Proposition 3.1]{Hay}.  Hence, or similarly, $q \in B_0^{\perp \perp}$ and is a closed projection there, and  $q B_0$ is closed.
\end{proof}

{\bf Remark.} Note that $B_0^{\perp \perp} \subset B^{**} \cap \{ q \}'$ clearly, however one can show that these sets differ in general.   Similarly, $q B_0^{\perp \perp} \neq q(B^{**} \cap \{ q \}') = q B^{**}q$.

\bigskip

The following contains the desired characterization of `peak interpolation-sets', as discussed 
in the second paragraph after Theorem \ref{peakthang22}.  Indeed (iii) is precisely the class of projections
discussed there, namely the projections corresponding to the (noncommutative version of the) extension property in the Bishop peak interpolation-set theorem.  
Item (vi) is a weaker version of this peak interpolation-set  extension property,  while         (v) is saying that $q$ is a noncommutative `interpolation-set'. 
Items (i), (ii), and (iv) are noncommutative reformulations of the classical condition $\mu_E = 0$ (or equivalently $\chi_E \, d \mu = 0$) for all measures $\mu \in A^\perp$ in Bishop's original theorem.
 We remark that  the `noncommutative null sets' in   (iv)   were first considered by Clou\^atre and Timko in \cite{CT}.

\begin{theorem} \label{peakthang}  {\rm (Noncommutative `Bishop peak interpolation-set' result)} \  Suppose that $A$ is a unital operator algebra, 
a subalgebra of a unital $C^*$-algebra $B$.
Suppose that  $q$ is a closed projection in $B^{**}$, and that 
 $B_0 = \{ b \in B : b q= qb \}$, and  $A_0 = A \cap B_0$.   The following are equivalent: 
\begin{itemize} \item [(i)] 
 $q B_0 \subset A_0^{\perp \perp}$.
 \item [(ii)] $A_0^{\perp} \subset (q B_0)_\perp$ 
 %ADD
 (or equivalently $A_0^{\perp} \subset q B^{**}$). 
 \item [(iii)]  If $b \in B_0$ 
 and $q b^* b q  \leq q d$ for an invertible positive $d \in B$ which commutes with $q$,
then  there exists an element $a \in A$ with $a q =  q a = b q$, and $a^* a \leq d$. 
 \item [(iv)]  $q$ is $A_0$-null (that is, if  $\varphi \in A_0^{\perp}$ then $|\varphi|(q) = 0$).
 \end{itemize} 
 If the above all holds, then $q$   is a $p$-projection for $A$ (or equivalently $q \in A^{\perp \perp}$, or equivalently $q$ is $A$-closed), $q A_0$ is a $C^*$-algebra,
 and we also have 
  \begin{itemize}   \item [(v)]   $q B_0 = q A_0$.
\item [(vi)] If $b \in B_0$ and $\| b q \|  \leq 1$ then  there exists an element $a \in {\rm Ball}(A)$ 
with $a q =  q a = b q$.
 \end{itemize} 
Item {\rm (vi)} implies {\rm (v)}.  If $q$   is a $p$-projection for $A$ then 
 items {\rm (i)}--{\rm (vi)} are all equivalent. \end{theorem}

\begin{proof}   Note that  $q B_0 \subset A_0^{\perp \perp}$ if and only if $\overline{q B_0}^{w*} = q B_0^{\perp \perp} \subset A_0^{\perp \perp}$;
and  if and only if  $A_0^{\perp} \subset (q B_0^{\perp \perp})_\perp$. 
The equivalence of (i) and (ii) follow from the bipolar theorem, or by taking upper and lower $\perp$'s.
Note that since $1$ is in $B_0$ (or $B_0^{\perp \perp}$), (i) and (ii) force $q \in A^{\perp \perp}$.   At the start of this section we discussed the equivalences between $B$-closed 
projections in $A^{\perp \perp}$, $A$-closed projections, and  $p$-projections for $A$.  

By \cite[Proposition 3.4]{Hay}, (ii) implies (v).
Conversely, suppose that $q \in A^{\perp \perp}$.  Then $q$ is
$A$-closed as we just said, and $q \in A_0^{\perp \perp}$ 
by Lemma \ref{qinz}.
Then (v) implies that $$q B_0 = q A_0 \subset A_0^{\perp \perp} A_0 \subset
A_0^{\perp \perp} .$$  Thus (i) holds.  Suppose in addition that $b \in B_0$ 
 and $q b^* b q  \leq q d$ for an invertible positive $d \in B$ which commutes with $q$.  Then there exists an element $a_0 \in A_0$ with $a_0 q = b q$, so that
 $q a_0^* a_0 q  \leq q d$.   
 By Theorem  \ref{peakthang22} there exists an element $a \in A_0$ with $a q = a_0 q = b q$, and $a^* a \leq d$.   
We have shown that (v) implies (iii) if $q \in A^{\perp \perp}$, and hence also that (ii) implies (iii) (since  (ii) implies (v) and $q \in A^{\perp \perp}$).   

That (iii) implies that  $q \in A^{\perp \perp}$ follows essentially from the principle that the Gamelin-Bishop theorem characterizes 
$p$-sets, or that  in the noncommutative case  the condition in the last sentence of Theorem  \ref{peakthang22} with $b = 1$ characterizes $p$-projections (see e.g.\ \cite[Theorem 5.10]{Hay}
for an even better result).    So by the above, (iii) implies (i). 
Clearly (iii) with $d = 1$ implies (vi). Conversely,  (vi) implies (v) by scaling.   Note that $q B_0$ is a $C^*$-algebra since it is the range of the map of multiplication by $q$, which is a 
$*$-homomorphism from $B_0$ into its bidual.

The equivalences involving (iv) may be deduced from Clou\^atre and Timko's proof of \cite[Theorem 6.2]{CT}, and follow from properties of the polar decomposition 
 of a linear functional as may be found in basic $C^*$-algebra 
texts (e.g.\ 3.6.7 in \cite{Ped} applied to the bidual).
  We include the proofs for completeness:
Suppose that $q$ is $A_0$-null  and $\varphi \in A_0^{\perp}$ with polar decomposition $\varphi  = u |\varphi |$ with $u \in B_0^{**}$.  
By Cauchy-Schwartz we have 
$$|\varphi(q b) |^2  =  |\varphi |(q b u ) ^2   \leq C |\varphi|(q) = 0 , \qquad b \in B_0 .$$
So (ii) holds.
Conversely if (ii) holds then, as we said above, we have (i) and (v) and also $q \in A_0^{\perp \perp}$.  
Suppose that $\varphi \in A_0^{\perp}$  with polar decomposition $\varphi  = u |\varphi |$  as above.  We have $q u^* \in q B_0^{**} 
 \subset A_0^{\perp \perp}.$
Thus  $|\varphi |(q) = \varphi(q u^*) = 0$.   So (iv) holds.  
  \end{proof}

 Note that $A_0 q$ is actually a $C^*$-algebra in the setting above, just as in the commutative case where $q = \chi_E$ for a closed set $E \subset K$, and $A_0 q = C(E)$.  
 Unfortunately $A q$ is not generally a $C^*$-algebra, which seems to be further evidence for the consideration of $A_0$ in place of $A$ in certain such results.
 
\begin{corollary} \label{peakthstr}  {\rm (Alternative noncommutative `Bishop peak interpolation-set' result in 
special cases)} \  Suppose that $A$ is a unital operator algebra, 
a subalgebra of a unital $C^*$-algebra $B$.
Suppose that  $q$ is a closed projection in $B^{**}$. 
As in the previous result (see also {\rm \cite[Theorem 6.2]{CT}}), $q B \subset A^{\perp \perp}$
if and only if $A^{\perp} \subset (q B)_\perp 
%ADD
= (q B^{**})_\perp$, and if and only if 
$q$ is $A$-null (that is,  if  $\varphi \in A^{\perp}$ then $|\varphi|(q) = 0$).  Moreover these conditions are equivalent to: 
 $qB = qA$ with $q \in A^{\perp \perp}$.  
If these conditions all  hold then so does  the `Bishop peak interpolation-set' result  in {\rm (iii)} 
of the previous theorem, 
under any one of the following extra hypotheses: 
\begin{itemize} \item [(1)] 
 $q \in B$.
 \item [(2)] $A + B_0$ is norm closed (or equivalently: every functional in $A_0^\perp$ extends to 
 a functional in $A^\perp$). 
 \item [(3)]  $q$ is central in $B^{**}$. 
 \item [(4)]   $q$ is  a minimal projection in $B^{**}$.
  \item [(5)]  $A$ is commutative.
  \end{itemize} 
\end{corollary}

\begin{proof}   The 
first two `if and only ifs' follow almost exactly as in the previous proof.   
Note that  $q B \subset A^{\perp \perp}$ if and only if $\overline{q B}^{w*} = q B^{**} \subset A^{\perp \perp}$, and  if and only if 
$A^{\perp} \subset (q B^{**})_\perp$. 
The equivalence with $q \in A^{\perp \perp}$ and $qB = qA$ is noted in  \cite[Theorem 6.2]{CT}.  Indeed one direction of this is obvious and as in the last proof: these conditions imply
$qB = qA \subset A^{\perp \perp} A \subset A^{\perp \perp}$.
The other direction is immediate from \cite[Proposition 3.4]{Hay}.   (There was a mistake in the statement and proof  of this equivalence 
  in the ArXiV version of \cite{CT}, which we pointed out to Clou\^atre and has been corrected in the published version.) 
Finally we show that each of  conditions (1)--(5) imply that the conditions in the last theorem are satisfied.

(2) \ If $A + B_0$ is norm closed then by a principle in functional analysis, $A^{\perp \perp} \cap B_0^{\perp \perp} = (A \cap B_0)^{\perp \perp} = A_0^{\perp \perp}$   
(this follows easily for example from \cite[Lemma I.1.14(a)]{HWW}).   
Hence $q B_0 \subset A^{\perp \perp}$ or $q B  \subset A^{\perp \perp}$ implies 
$$q B_0 \subset A^{\perp \perp} \cap B_0^{\perp \perp} = A_0^{\perp \perp}.$$ 
So the conditions in the last theorem are satisfied.    The other equivalence here is no doubt also a principle in functional analysis known in some 
quarters, which we now explain.   Indeed if the restriction map from $B^*$ to $B_0^*$ maps $A^\perp$   onto $A_0^\perp$,
then it is the dual of a bicontinuous injection $(A_0^\perp)_* = B_0/A_0$ into 
$(A^\perp)_* = B/A$.       This injection is the canonical map $b_0 + A_0 \mapsto b_0 + A$,
since the dual of the latter is the restriction map.    
This argument is reversible: if the canonical map $b_0 + A_0 \mapsto b_0 + A$ is bicontinuous then 
 the restriction map from $B^*$ to $B_0^*$ maps $A^\perp$   onto $A_0^\perp$.
By the closed range theorem applied to the canonical map $A+B_0 \to (A+B_0)/A \subset B/A$, if $A+B_0$ is closed then $(B_0 + A)/A$ is 
isomorphic to $B_0/A_0$ (see \cite[Lemma I.1.14]{HWW} if needed).   Conversely, if the canonical map $b_0 + A_0 \mapsto b_0 + A$
is bicontinuous then $(B_0 + A)/A$ is closed in $B/A$, so that $A+B_0$ is closed. 

(1) \ If $q \in B$ and $q B  \subset A^{\perp \perp}$  then  $q \in B \cap A^{\perp \perp} = A$.  Then 
 $$q B_0 \subset 
q A^{\perp \perp} q =  (qAq)^{\perp \perp} = \{ a \in A : a = qa = aq \}^{\perp \perp} 
\subset A_0^{\perp \perp}.$$ 
So again the conditions in the theorem are satisfied.    Alternatively, 
we can apply  \cite[Proposition 3.4]{Hay}  and \cite[Theorem 3.4]{Bnpi}  to $B_0$ and $qAq$.

(3) \ This is obvious e.g.\ from Theorem \ref{peakthang}.  

 (4) \ We have $q \in A^{\perp \perp}$
and so $q B_0 = \Cdb q \subset A_0^{\perp \perp}$ by  the lemma. 

(5) \ Suppose that $A$ is commutative and $A^{\perp} \subset (q B)_\perp$.   Then $q \in A^{\perp \perp}$ as above and the latter is also commutative.  So $A = A_0$
and $A_0^{\perp} \subset (q B)_\perp \subset (q B_0)_\perp$.
This does it.  Alternatively, by the argument towards the end of the paragraph after Theorem \ref{peakthang} there exists $a \in A$ with $aq = qa = bq$ 
and we may conclude as in that argument. 
\end{proof}

{\bf Remarks.} \ 1) \ In Bishop's result $A$ need not be an algebra.   One may therefore hope to extend some of our results above to the  unital operator space setting.   Note that 
\cite[Proposition 3.4]{Hay}, an ingredient above, does not need $A$ to be an algebra.

2)	\ One approach to replacing $A_0$ by $A$, is to try to find conditions under which $\varphi \in A_0^\perp$ 
implies that there is an extension $\tilde{\varphi} \in A^\perp$.     If this holds and if  $A^\perp \subset 
(q B)_\perp$, then it is easy to check condition (ii) in Theorem \ref{peakthang} directly.   Indeed suppose that $\varphi \in A_0^\perp$.   
If it has an extension $\tilde{\varphi} \in A^\perp$ then 
by hypothesis $\tilde{\varphi} \in (q B)_\perp$, so that $\varphi \in (q B_0)_\perp$ (using also the lemma above).  

We thank Ken Davidson for suggesting to us in the discussions mentioned earlier that perhaps there always exists an extension 
of functionals in $A_0^\perp$ to $A^\perp$,
if and only if  $A+B_0$ is closed.

3) \ We  do not know if $qB  \subset A^{\perp \perp}$ implies the conditions in  Theorem \ref{peakthang} in all cases, nor do we have a counterexample at this time.  
By symmetry if $q$ is $A$-null then we have both $qB  \subset A^{\perp \perp}$  and $Bq  \subset A^{\perp \perp}$ .
A main difficulty that arises is that by the argument towards the end of the paragraph after Theorem \ref{peakthang} 
one may obtain `left interpolating' elements $a_1$ and `right interpolating' elements $a_2$.  We have  
$a_1 q = q a_2$, but one really needs $a_1 = a_2$ and we have not been able to spot the trick to ensure this (except under strong hypotheses).

4) \ It seems possible that the ideas in \cite[Corollary 5.4]{BRII} (taking $b$ there in $B$) might give another noncommutative   variation of Bishop's peak interpolation-set theorem.

5) \ In many cases discussed in \cite{CD,CD2,DH} the algebra $A$ is commutative so that (5) in the last theorem applies.

\end{document}